\documentclass[12pt]{amsart}
\usepackage{mathrsfs}
\usepackage{amssymb}
\usepackage{amsthm}
\usepackage{tikz}
\usepackage{titletoc}
\usepackage{amssymb}
\usepackage{amsxtra}
\usepackage{mathrsfs}
\usepackage{amsmath,amstext,amsfonts,verbatim, url}
\usetikzlibrary{decorations.pathreplacing}
\newcommand{\diam}{\textrm{diam}}

\newcommand{\var}{\textrm{var}}
\newtheorem{theorem}{Theorem}
\newtheorem{corollary}{Corollary}
\newtheorem{prop}{Proposition}

\newtheorem{thmx}{Theorem}

\newtheorem{lemma}{Lemma}
\newtheorem{remark}{Remark}
\newtheorem*{thm}{Theorem {\bf 1$'$}}
\title{Multifractal analysis  in non-uniformly
hyperbolic interval maps}

\author{Ma Guanzhong, Shen Wenqiang$^*$, Yao Xiao}
\address{$^1$ School of Mathematics and Statistics, Anyang Normal University, Anyang, 455000, P. R. China.
\vskip 2pt $^2$ School of mathematics and statistics, Northwestern Polytechnical University, Xi'an, 710000, P. R. China.
\vskip 2pt $^3$ School of Mathematical sciences and LPMC, Nankai University, Tianjin, 300071, P.R. China.
\vskip 2pt \hspace{1.5mm} Email:{maguanzhong75@aynu.edu.cn\\ wqshen@nwpu.edu.cn\\ yaoxiao@nankai.edu.cn}}
\thanks{{$^*$ Corresponding author}}

\begin{document}
\maketitle
\begin{abstract}
In this paper, we study the Hausdorff dimension of the generalized intrinsic level set with respect to the given ergodic meausre in a class of
non-uniformly hyperbolic interval maps with finitely many branches.
\end{abstract}
\smallskip


\thanks{{ Keywords:  Multifractal analysis; Hausdorff dimension; Moran set}
\section{Introduction}

Let $T:\bigcup\limits_{i=1}^m I_i\subset [0, 1]\rightarrow [0,1]$ be a piecewise $C^{1}$ map, where $I_i$ is the closed interval for  $1\leq i\leq m$ such that   $int(I_{i})\cap int(I_{j})=\emptyset$ for any distinct $i$ and $j$. Here, $int(I_i)$ means the interior of $I_i$.  In this paper, we consider the following class of  non-uniformly hyperbolic interval maps,
\begin{itemize}
\item $T|_{ I_i}:I_i\rightarrow [0,1]$ is a surjective and  continuously differentiable for $1\leq i\leq m.$ There is a unique $x_i\in I_i$ such that $T(x_i)=x_i$ for each $i$.
\item   $|T'(x)|>1$ for $x\not\in\{x_1,\ldots, x_m\}$. Here, we also allow that $|T'(x_i)|>1$ for some $i\in \{1, 2, \ldots, m\}$.
\end{itemize}
If $T'(x_i)=1$ or $T'(x_i)=-1$ for some $i$, we say that  $x_i$ is a {\it parabolic} fixed point.   The  class of non-uniform hyperbolic maps includes the important  example of Manneville-Pomeau map \cite{PW1999}, $T: [0,1]\rightarrow [0,1]$ defined by $Tx= x+x^{1+\beta}$ (mod 1), where $0<\beta<1$, see Fig 1.

\begin{figure}
\begin{center}
\begin{tikzpicture}[scale=6]
    \node[left]at(0,1){$1$};
    \node[below]at(1,0){$1$};
    \node[below left]at(0,0){$0$};

    \draw (0,0)--(1,0)--(1,1)--(0,1)--cycle;
    \draw[domain=0:0.6180339887498949,color=red] plot (\x,{\x+(\x)^2}) ;
    \draw[domain=0.6180339887498949:1,color=blue] plot (\x,{\x+(\x)^2-1}) ;
    \draw[densely dashed](0.6180339887498949,0)--(0.6180339887498949,1);
\end{tikzpicture}
\caption{Maneville map}
\end{center}
\end{figure}
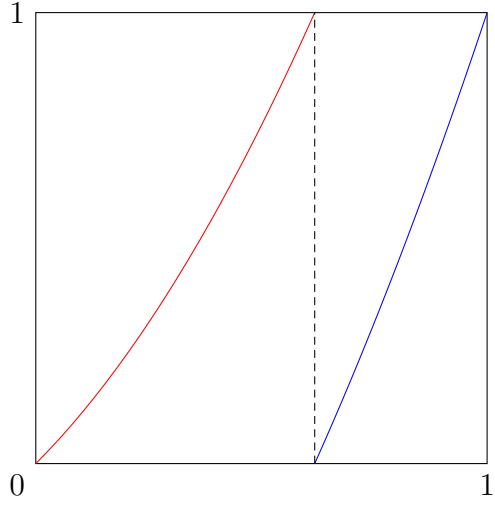

We define the repeller $\Lambda$ of $T$ by
$$
  \Lambda:=\left\{x\in \bigcup_{i=1}^m I_i : T^{n}(x)\in [0,1], \forall n\geq0\right\}.
$$
 We know that $(\Lambda, T)$ has a very natural Markov partition.  Let  $S_{i}$ be the inverse branch of $T|_{I_{i}}: I_{i}\rightarrow [0,1]$ for $i=1,\ldots,m$.
Let $\mathcal{A}=\{1\dots, m\}$, $\Sigma=\mathcal{A}^{\mathbb{N}}$ and $\Sigma_n$ be the set of all $n$-blocks over $\mathcal{A}$ for any $n\in \mathbb{N}$. Here we point out $\mathbb{N}$ is the set of positive integers in this paper. Let $\sigma:\Sigma\to\Sigma$ be the shift map,
   $$\sigma((\omega_n)_{n\ge 0})=(\omega_n)_{n\ge 1}.$$ Define the semi-conjugacy map $\pi: \Sigma \to [0,1]$ by
\[
\pi(\omega):=\lim_{n\rightarrow\infty}S_{\omega_0}\circ S_{\omega_2}\circ \dots \circ S_{\omega_{n-1}}([0,1]).
\]
It is easy to check that $\pi(\Sigma)=\Lambda$ and $ \pi\circ\sigma(\omega)=T\circ\pi(\omega).
$ We remarked that $\pi$ is a bijection map from  $\Sigma$ to $\Lambda$ except for a countable many points.

Let $\phi$ be a continuous function in  $C(\Lambda, \mathbb{R}^{d})$. We denote the $n$-th Birkhoff sum by
$S_n\phi(x)=\sum_{j=0}^{n-1}\phi(T^{j}x)$ and $n$-th Birkhoff average by
$A_n \phi(x)=\frac{1}{n} S_n\phi(x)$ for any $x\in \Lambda$.  For any $\alpha\in \mathbb{R}^{d}$, we define
\begin{equation*}
\Lambda_{\phi}(\alpha):=\{x \in\Lambda: \lim_{n\rightarrow\infty} A_n\phi(x)=\alpha\}
\end{equation*}
and more generally,
\begin{equation*}
\tilde{\Lambda}_{\phi}(\alpha):=\{x\in\Lambda: \lim_{k\rightarrow\infty} A_{n_k}\phi(x)=\alpha  \ \text{for some} \  \{n_k\}_{k=1}^{\infty} \ \text{with} \lim_{k\rightarrow\infty}n_k=\infty\}.
\end{equation*}

Since $\Lambda_{\phi}(\alpha)$ not only depends on $\alpha$, but also on the continuous  function $\phi$, it is natural to introduce a set which is intrinsic in some sense.

We denote the set of all invariant measures by $\mathcal{M}(\Lambda, T)$ and the set of all ergodic measures by $\mathcal{E}(\Lambda, T)$.
For any $\mu\in \mathcal{E}(\Lambda, T)$, we define the intrinsic level set
\begin{equation*}
\Lambda_{\mu}:=\left\{x\in  \Lambda: (A_n)_{*}\delta_x=\frac{1}{n}\sum_{i=0}^{n-1}\delta_{T^i(x)}
  \xrightarrow[n\rightarrow\infty]{*} \mu \right\}
\end{equation*}
be the set of all  $\mu$ generic points, where $\xrightarrow[n\rightarrow\infty]{*}$ stands for the convergence in the weak-$*$ topology. Similarly, we also define the generalized intrinsic level set
\begin{equation*}
\tilde{\Lambda}_{\mu}:=\left\{x\in  \Lambda: (A_{n_k})_{*}\delta_x
  \xrightarrow[k\rightarrow\infty]{*} \mu \ \text{for some} \  \{n_k\}_{k=1}^{\infty} \ \text{with} \lim_{k\rightarrow\infty}n_k=\infty\right\}.
\end{equation*}

The  multifractal analysis for uniformly hyperbolic conformal dynamical system is well developed in past years,  we refer the \cite{PW1997A, GP1997, FLW2002, Olsen-2003, Olsen-2007} for  entropy spectrum and Birkhoff spectrum of level set and  dimension spectrum of Gibbs measure or weak Gibbs measure. The original method developed in \cite{PW1997A} is based on the equilibrium state (Gibbs measure) in thermal-formalism method, which needs the further regularity conditions on  $T$. Later, this method was further developed in \cite{FLW2002, Olsen-2003, Olsen-2007, Climenhaga}, and the  regularity conditions have been completely removed in the uniformly hyperbolic case.

Recently, there has been a trend  in  understanding the multifractal analysis beyond the uniformly hyperbolic dynamical system. However, up to now,  there is still not a complete picture in the direction of non-uniformly hyperbolic dynamical system. The topological entropy of these types of sets have been studied in \cite{PS2005, PS2007,  Tian-Varandas,  Climenhaga-2013,  zheng-chen-zhou, Takens-Verbitskiy}. The dimension spectrum of Birkhoff
ergodic limits in non-uniform hyperbolic dynamic systems has been done in \cite{JJOP2010, BPM2014, GR20092, Chung2010, Climenhaga, Lommi-Todd}.

In \cite{GR20091},  Gelfert and Rams  studied the dimension spectrum of $\Lambda_{\phi}$ for $\phi=\log |T'|$ by. For the general continuous function $\phi$,  Johansson, Jordan,  \"{O}berg, and Pollicott established a formula of $\dim_{H}\Lambda_{\phi}$ in \cite{JJOP2010}.

For any ergodic measure $\mu$, we call $\mu$ is hyperbolic if the Lyapunov exponent of $\mu$
$$\lambda(\mu, T):=\int \log |T'|d\mu>0,$$
 otherwise  $\mu$ is called to be parabolic.  If $\mu$ is a parabolic  measure, $\mu$ only supports on some parabolic fixed point under the setting of this paper.

Given a compact and $T$-invariant  set $K\subset \Lambda$, we say that $K$ is a hyperbolic set, if $|T'(z)|>1$ for any $z\in K$.  The hyperbolic dimension of $\Lambda$ is defined by $$\dim_{H}^{hyp}\Lambda:=\sup\{\dim_{H}K: K  \ \text{is a hyperbolic set in} \ \Lambda\}.$$ One of the main goal in the study of non-uniformly hyperbolic dynamical system is to recover the sufficient hyperbolicity to dominate or balance the non-hyperbolic behaviour, we refer the reader to \cite{Dobbs-2006, Dobbs-2015}.

Now, we will state our main result of the Hausdorff dimension of $$\Lambda_\mu=\left\{x\in  \Lambda: \frac{1}{n}\sum_{i=0}^{n-1}\delta_{T^i(x)}  \xrightarrow{*}\mu\right\}$$
and
$$
\tilde{\Lambda}_{\mu}=\left\{x\in  \Lambda: (A_{n_k})_{*}\delta_x
  \xrightarrow[n\rightarrow\infty]{*} \mu \ \text{for some} \  \{n_k\}_{k=1}^{\infty} \ \text{such that} \lim_{k\rightarrow\infty}n_k=\infty\right\}.
$$

The topological entropy of $h_{top}(\Lambda_{\mu})$ and $h_{top}(\tilde{\Lambda}_{\mu})$ can be deduced from \cite{PS2007, FLP}. The following result proved by Pfister and Sullivan in \cite{PS2007}, and Fan, Liao, and Peyri\`ere \cite{FLP} independently for the system with the (almost) specification property. It is known that this property holds in  the non-uninformly  hyperbolic interval map we considered here.

\begin{thmx}[\cite{PS2007, FLP}]
Let $m$ be an invariant measure in $\mathcal{M}(\Lambda, T)$.  We have
\begin{equation}
h_{top}(\tilde{\Lambda}_{m})=h_{top}(\Lambda_{m})=h(m, T).
\end{equation}
Especially,
$$h_{top}(\tilde{\Lambda}_{m})=0$$ if $\lambda(m, T)=0$.
\end{thmx}

Johannes and Takahasi \cite{JT2020} study  the Hausdorff dimension of $\Lambda_{\mu}$ in the setting of non-uniformly hyperbolic interval maps even with infinitely many branches recently.
\begin{thmx}[\cite{JT2020}]\label{thmx-JT2020}
 Let $m$ be an ergodic measure in $\mathcal{M}(\Lambda, T)$.  We have
\begin{enumerate}
\item $\dim_{H}\Lambda_{m}=\frac{h(m, T)}{\lambda(m, T)}$ if $\lambda(m, T)>0$;
\item $\dim_{H}\Lambda_{m}\geq \dim^{hyp}_{H}\Lambda>0$ if $\lambda(m, T)=0$. Moreover if $T$ is a $C^2$-map, we have
$$\dim_{H}\Lambda_{m}= \dim_{H}\Lambda.$$
\end{enumerate}
\end{thmx}

To the best of our knowledge, the relation between $\dim_{H}\Lambda_{m}$ and $\dim_{H}\tilde{\Lambda}_{m}$  is still lacking even in the setting of the finitely many branches. We have the following result in this direction.

\begin{theorem}\label{theorem-main}
 Let $m$ be an ergodic measure in $\mathcal{M}(\Lambda, T)$.  We have
\begin{enumerate}
\item $\dim_{H}\tilde{\Lambda}_{m}=\frac{h(m, T)}{\lambda(m, T)}$ if $\lambda(m, T)>0$;
\item $\dim_{H}\tilde{\Lambda}_{m}\geq\dim_{H}\Lambda_{m}\geq \dim^{hyp}_{H}\Lambda>0$ if $\lambda(m, T)=0$. Moreover if $T$ is a $C^2$-map, we have
$$\dim_{H}\tilde{\Lambda}_{m}=\dim_{H}\Lambda_{m}= \dim_{H}\Lambda.$$
\end{enumerate}
\end{theorem}
\begin{remark}
One sees that there is a great difference on the size of intrinsic level set of a invariant measure from the viewpoints of topological entropy and Hausdorff dimension.
In fact, the $C^2$ regularity condition is only used to ensure that
$$\dim_{H}^{hyp}\Lambda=\dim_{H}\Lambda.$$

Indeed, the assumption of the ergodicity is not necessary in Theorem \ref{theorem-main}. This can be removed by the careful approximating arguments and we do not address it  in this paper.
\end{remark}
By the same technique in the proof of Theorem \ref{theorem-main}, we can get the following result.
\begin{thm}\label{thm-generalization-2}
Let $K$ be a compact and connected set in $ \mathcal{M}(\Lambda, T)$ and
$$\Lambda_{K}:=\{x\in \Lambda: \rm {Asym}(\{(A_n)_{*}\delta_{x}\}_{n=1}^{\infty})=K\},$$
where $\rm {Asym}(\{(A_n)_{*}\delta_x\}_{n=1}^{\infty})$ is the set of all accumulating points of $\{(A_n)_{*}\delta_x\}_{n=1}^{\infty}$ in $\mathcal{M}(\Lambda, T)$.
Then we have
\begin{enumerate}
\item $\dim_{H}\Lambda_{K}=\inf \left\{\frac{h(m, T)}{\lambda(m, T)}: m\in K, \lambda(m, T)>0\right\}$ if $K$ has some invariant measure $\mu$ with $\lambda(\mu, T)>0$;
\item $\dim_{H}\Lambda_{K}=\dim_{H}\Lambda$ if $\lambda(m, T)=0$ for any $m\in K$ and  $T$ is a $C^2$ map.
\end{enumerate}
\end{thm}

Here we do not pursue this generalization in this paper. Motivated by Theorem \ref{theorem-main} and Theorem \ref{thmx-JT2020}, we have the following corollary immediately.
\begin{corollary}
Let $m$ be an ergodic measure and $T$ is $C^2$ in the setting above. Then we have
\begin{equation*}
\dim_{H}\tilde{\Lambda}_m=\dim_{H}\Lambda_m.
\end{equation*}
\end{corollary}

It is proved that in \cite{FLW2002} that $\dim_{H}\tilde{\Lambda}_m=\dim_{H}\Lambda_m$ holds for the $C^1$ conformal repeller. We do not known whether this holds or not in the uniformly hyperbolic case under the $C^1$ condition.  For this direction, we have the following result.

\begin{theorem}\label{thm-uni-parabolic}
Let $m$ be an ergodic measure and assume that $T$ has only one parabolic fixed point $p$ in $[0, 1)$. Then we have
\begin{equation*}
\dim_{H}\tilde{\Lambda}_m=\dim_{H}\Lambda_m.
\end{equation*}
\end{theorem}

If $T$ has more than 1 fixed points, we do not know how to control  the persistent recurrence behaviors between (or among) the multiple parabolic fixed points in the  $C^1$ regularity condition.

%
%
%

In the proof of Theorem \ref{theorem-main}, we also recover a complete proof of Theorem B by using our framework for the reader's convenience. In \cite{JT2020}, a proof of Theorem \ref{thmx-JT2020} was already given even in the framework of infinitely branches by  a series of delicate approximation arguments.
Here, we will give a direct proof in our simple setting, which is inspired by \cite{FLP, JJOP2010}.

This paper is organised as follows. In Section 2, some some basic preliminaries and results are  collected. In Section 3, we will prove the upper bound of Theorem \ref{theorem-main}. In Section 4, we will make some effort to give a unified framework of geometric Moran construction in dimension 1, which may be of independent interest. We  believe that this framework can be used to deal with many problems in multifractal analysis  in non-uniformly hyperbolic dynamical system. In Section 5, we will prove the lower bound for Theorem 1. Finally, Theorem 2 will be proved in Section 6.
\section{Preliminaries}
In this section, we will collect some basic facts. For any $\omega\in \Sigma$, we denote $D_n(\omega)=\max\{|x-y|: x, y\in S_{\omega_0}\circ S_{\omega_1}\cdots S_{\omega_{n-1}}([0, 1])\}$.

 Let $g(x)=\log |T'(x)|$ for any $x\in \Lambda$ and $G(\omega)=\log |T'(\pi(\omega))|$ for any $\omega\in \Sigma$. And we know that $g(\pi(\omega))=G(\omega)$ for $\omega\in \Sigma$. For any $f \in C(\Lambda, \mathbb{R})$, we denote $F: \Sigma\rightarrow \mathbb{R}$ by
\begin{equation}
F(\omega)=f(\pi(\omega)).
\end{equation}
For any $n\in \mathbb{N}$, we define
\begin{equation}
\var_n(f):=\sup\{|f(\pi(\omega)-f(\pi(\tau))|:\omega_0=\tau_0, \dots, \omega_{n-1}=\tau_{n-1}\}.
\end{equation}

The following lemma shows a relation between the logarithm growth of $D_n(\omega)$ and the Birkhoff average of $G(\omega)$.
\begin{lemma}{\cite{JJOP2010}}\label{lemma-growth-1}
Under the setting above, we have
\begin{equation*}
\lim_{n\rightarrow\infty}\sup_{\omega\in \Sigma}\Big|-\frac{\log D_n(\omega)}{n}-A_nG(\omega)\Big|=0.
\end{equation*}
\end{lemma}

In our following discussion, we also need the following fact.
\begin{lemma}{\cite{JJOP2010}}
Let $f$ be a continuous function on $\Lambda$. Then we have
\begin{equation*}
\lim_{n\rightarrow\infty}\frac{1}{n}\rm{\var}_n(f)=0.
\end{equation*}
\end{lemma}

Since $\Lambda$ is a compact set, there exists a countable dense subset $\{f_n\}_{n=1}^{\infty}$ in $C(\Lambda, \mathbb{R})$ and we can assume $f_n\not\equiv 0$ for any $n\in \mathbb{N}$ without loss of generality. Let $\mathcal{M}(\Lambda)$ be the set of all probability measures. We introduce
a metric $d$ on $\mathcal{M}(\Lambda)$ by
$$d(\mu, \nu):=\sum_{n=1}^{\infty}\frac{|\int f_n d \mu-\int f_n d\nu|}{2^{n}\|f_n\|}.$$ And the topology induced by the metric $d$ on $\mathcal{M}(\Lambda)$ is compatible
with the weak-$*$ topology.

\begin{lemma}\label{decomposition-level-lemma}
Let $m$ be an invariant measure in $\mathcal{M}(\Lambda, T)$ and $$\Lambda_{m, k}:=\{x\in \Lambda: \lim\limits_{n\rightarrow\infty}A_nf_i(x)=\int f_i d m, \ \text{for} \ 1\leq i\leq k\}.$$ Then we have
\[\Lambda_{m}=\bigcap_{k=1}^{\infty}\Lambda_{m, k}.\]
\end{lemma}
The proof of Lemma \ref{decomposition-level-lemma} is rather simple, we will omit the proof here. The existence of parabolic fixed points has a important impact on the size of $m$-generic set. Assume that $T$ has
$l+1$ different parabolic fixed points $\{p_i, 0\leq i \leq l\}$ and we denote
$$\mathcal{M}^{p}(\Lambda, T):=\left\{\sum_{i=0}^{l}\lambda_i\delta_{p_i}: \lambda_i\geq 0\ \text{and} \ \sum_{i=0}^{l}\lambda_{i}=1\right\}.$$
It is obvious that $\mathcal{M}^{p}(\Lambda, T)$ is a $l$-dimensional simplex in $\mathcal{M}(\Lambda, T)$ and
$$
\mathcal{M}^{p}(\Lambda, T)=\{\mu\in \mathcal{M}(\Lambda, T): \lambda(\mu, T)=0\}.
$$

\begin{lemma}\label{lemma-parabolic}
Let $\mu$ be a hyperbolic measure in $\mathcal{M}(\Lambda, T)$ and
$$\mathcal{M}_{\mu, k}:=\left\{\nu\in \mathcal{M}(\Lambda, T): \int f_i d\nu=\int f_i d\mu \ \text{for} \ 1\leq i\leq k\right\}.$$
Then we have
$$
\mathcal{M}_{\mu, k}\cap \mathcal{M}^{p}(\Lambda, T)=\emptyset
$$
for $k$ sufficiently large.
\end{lemma}
\noindent{\bf Proof of Lemma \ref{lemma-parabolic}}  We assume that
there exists a sequence $\{k_i\}_{i=1}^{\infty}$, such that
$$
\mu_{k_i}\in\mathcal{M}_{\mu, k_i}\cap \mathcal{M}^{p}(\Lambda, T).
$$
It follows that
\[
\int f_jd \mu_{k_i}=\int f_j d \mu
\]
for any $j\leq k_i$, which implies that
\[d(\mu_{k_i}, \mu)\leq \frac{1}{2^{k_i}}\]
and  $$\lim_{i\rightarrow\infty}\mu_{k_i}=\mu.$$
Recalling that $\lambda(\mu_{k_i}, T)=0$ for any $i\in \mathbb{N}$, we get $\lambda(\mu, T)=0$, which contradicts the fact that $\mu$ is hyperbolic.
\hfill $\square$

\section{Proof for the upper bound}

Recall that $g(x)=\log |T'(x)|$ for any $x\in \Lambda$. We define
\[I_n(\omega):=S_{\omega_0}\circ S_{\omega_2}\circ \dots \circ S_{\omega_{n-1}}([0,1])\]
for $\omega\in\Sigma$.
For any  $\delta>0$, $\epsilon\in (0, \delta)$, $\alpha=(\alpha_1, \dots, \alpha_k)\in \mathbb{R}^{k}$  and $n\in \mathbb{N}$,
let  $G_{k}(\alpha, \delta; n, \epsilon)$ be the set of all cylinders  in $\Sigma_n$ such that for any $[\omega]_n\in G_{k}(\alpha, \delta; n, \epsilon)$    there exists $x\in I_n(\omega)$ satisfying the following properties
\begin{enumerate}
\item $\Big|\frac{1}{n}S_n\phi_{i}(x)-\alpha_i\Big|<\epsilon$ for $1\leq i \leq k$;
\item $A_ng(x)\geq \delta-\epsilon$.
\end{enumerate}

We define the two pressure like functions $g_{k}(\alpha, \delta; s, n, \epsilon)$
and $g_{k}^{*}(\alpha, \delta; s, n, \epsilon)$  associated to $G_{k}(\alpha, \delta; n, \epsilon)$ as follows,
\begin{equation}
g_{k}(\alpha, \delta; s, n, \epsilon):=\sum_{[\omega]_n\in G_{k}(\alpha, \delta; n, \epsilon) }(\diam I_{n}(\omega))^s
\end{equation}
and
\begin{equation}
g_{k}^{*}(\alpha, \delta; s, n, \epsilon):=\sum_{[\omega]_n\in G_{k}(\alpha, \delta; n, \epsilon) }\sup_{x\in I_n(\omega)}\exp(-s S_n g(x))
\end{equation}

Let
\begin{equation}
\overline{g}_{k}(\alpha, \delta; s):=\lim_{\epsilon\rightarrow 0}\limsup_{n\rightarrow\infty}\frac{1}{n}\log g_{k}(\alpha, \delta; s, n, \epsilon)
\end{equation}
and
\begin{equation}
\underline{g}_{k}(\alpha, \delta; s):=\lim_{\epsilon\rightarrow 0}\liminf_{n\rightarrow\infty}\frac{1}{n}\log g_{k}(\alpha, \delta; s, n, \epsilon).
\end{equation}
Similarly, we define $\underline{g}^*_{k}(\alpha, \delta; s)$ and $\overline{g}^*_{k}(\alpha, \delta; s)$. Then we have  the following Lemma.
\begin{lemma}
Under  the  setting above,  we have
\begin{equation}
\underline{g}_{k}(\alpha, \delta; s)=\overline{g}_{k}(\alpha, \delta; s)=\underline{g}_{k}^{*}(\alpha, \delta; s)=\overline{g}_{k}^{*}(\alpha, \delta; s)
\end{equation}
for any $s\geq 0$.
\end{lemma}
\begin{proof}
The proof can be essentially deduced from \cite{FLW2002}.
\end{proof}
We denote the common limit above by $g_{k}(\alpha, \delta; s)$.
Recalling that there exists a dense set  $\{f_n\}_{n=1}^{\infty}$ in $C(\Lambda, \mathbb{R})$ and $f_n\not\equiv 0$ for any $n\in \mathbb{N}$.
Denote $F_{k}: \Lambda\rightarrow \mathbb{R}^{k}$ by
\begin{equation}
F_{k}(x)=(f_1(x), \dots, f_k(x)).
\end{equation}
And for any $\mu\in \mathcal{M}(\Lambda, T)$, we define
\begin{equation}
\int F_k d\mu: =\bigg(\int f_1 d\mu, \dots, \int f_k d\mu\bigg).
\end{equation}

For any $\alpha=(\alpha_1, \dots, \alpha_k)$, we introduce the norm $\|\cdot\|_{\infty}$  on $\mathbb{R}^{k}$ by
\begin{equation}
\|\alpha\|_{\infty}:=\sup_{1\leq i\leq k}|\alpha_i|.
\end{equation}

\begin{prop}\label{prop-key}
Let $m$ be an invariant measure in $\mathcal{M}(\Lambda, T)$ such that $\lambda(m, T)>0$ and $\alpha=\int F_k d m\in \mathbb{R}^{k}$.
For any $\delta\in (0, \lambda(m, T))$ and $\tau\in (0, \delta)$, we have
\begin{equation}
g_{k}(\alpha, \delta; s)\leq \sup\left\{h(\mu, T)-s\lambda(\mu, T): \mu\in
X(\alpha, \delta, k, \tau)
\right\},
\end{equation}
where
$$
X(\alpha, \delta, k, \tau):=\left\{\mu\in \mathcal{M}(\Lambda, T): \left\|\int F_k d\mu-\alpha\right\|_{\infty}\leq \tau, \lambda(\mu, T)\geq \delta-\tau\right\}.
$$
\end{prop}

\begin{proof}
Let $a=g_{k}(\alpha, \delta; s)$. We have
$$
a=\lim_{\epsilon\rightarrow 0}\limsup_{n\rightarrow\infty}\frac{1}{n}\log\left(\sum_{[\omega]_n\in G_{k}(\alpha, \delta; n, \epsilon) }\sup_{x\in I_n(\omega)}\exp(-s S_n g(x))\right).
$$
We first introduce a probability measure $\tilde{\mu}_n$ on $\Sigma_n$. For any $[\omega]_n\in G_{k}(\alpha, \delta; n, \epsilon) $, we define
\[
\tilde{\mu}_n(\omega):=\frac{\sup\limits_{x\in I_n(\omega)}\exp(-sS_n g(x))}{\sum\limits_{[\omega]_n\in G_{k}(\alpha, \delta; n, \epsilon) }\sup\limits_{x\in I_n(\omega)}\exp(-sS_n g(x))}
\]
Now, we construct a $n$-Bernoulli measure $(\tilde{\mu}_n)^{\otimes}$ on $\Sigma$, which is just the products of countably many copies of $(\Sigma_n, \tilde{\mu}_n)$. Here, we identify the probability spaces $((\Sigma_n)^{\mathbb{N}}, (\tilde{\mu}_n)^{\mathbb{N}})$ with $(\Sigma, (\tilde{\mu}_n)^{\otimes})$ by the following isomorphism $\Theta: ((\Sigma_n)^{\mathbb{N}}, (\tilde{\mu}_n)^{\mathbb{N}})\rightarrow  (\Sigma, (\tilde{\mu}_n)^{\otimes})$
\begin{equation}
\Theta (\theta_1, \theta_2,  \dots, \theta_n, \dots)=(\theta_1\theta_2\dots\theta_n\dots)
\end{equation}
for any $(\theta_1, \dots, \theta_n, \dots)\in ((\Sigma_n)^{\mathbb{N}}, (\mu_n)^{\mathbb{N}})$.
It is easy to see that $(\tilde{\mu}_n)^{\otimes}$ is ergodic measure on $(\Sigma, \sigma^n)$. We refer the interesting readers to \cite{JJOP2010} for a very good introduction of the $n$-Bernoulli measure. It follows that
\begin{align*}
h((\tilde{\mu}_n)^{\otimes}, \sigma^n)=H(\tilde{\mu}_n)
= s\int_{\Sigma_n}\inf_{x\in I_n(\omega)}S_n g(x)d\mu_n(\omega)+I_n,
\end{align*}
where
$$I_n:=\log\left(\sum_{[\omega]_n\in G_{k}(\alpha, \delta; n, \epsilon) }\sup_{x\in I_n(\omega)}\exp(-s S_n g(x))\right).
$$

Let $\mu_{n}^{\otimes}=\pi_{*}(\tilde{\mu}_n)^{\otimes}$ and $\nu_n=(A_{n})_{*}\mu_{n}^{\otimes}$. We have $\nu_n\in \mathcal{M}(\Lambda, T)$ and
\begin{align*}
& h(\nu_n, T)-s\int\log|T'(x)|d\nu_n(x)
\geq & -\frac{|s|}{n}\sup_{[\omega]_n}\sup_{x, y\in I_n(\omega)}|S_n g(x)-S_n g(y)|+\frac{I_n}{n}.
\end{align*}

Also, it is not difficult to prove that for $n$ sufficiently large, we have
\begin{equation}
\int g d\nu_n=\int A_n g d\mu_{n}^{\otimes}\geq \delta -2\epsilon
\end{equation}
and
\begin{equation}
\Big|\int F_{k}d\nu_n-\alpha\Big|_{\infty}\leq 2\epsilon.
\end{equation}
Then for any $\epsilon<\frac{\tau}{2}$, we have $\nu_n\in\in
X(\alpha, \delta, k, \tau)$ and
\begin{align*}
&\sup\left\{h(\mu, T)-s\lambda(\mu, T): \mu\in
X(\alpha, \delta, k, \tau)\right\}\\
\geq & \limsup_{n\rightarrow\infty}\left(h(\nu_n, T)-s\int\log|T'(x)|d\nu_n(x)\right)\\
\geq & \limsup_{n\rightarrow\infty} \frac{1}{n}\log\left(\sum_{[\omega]_n\in G_{k}(\alpha, \delta; n, \epsilon) }\sup_{x\in I_n(\omega)}\exp(-s S_n g(x))\right).
\end{align*}
Taking $\epsilon \rightarrow 0$, we complete the proof of Proposition \ref{prop-key}.
\end{proof}

\begin{proof}

Now we are going to prove  the upper bound
\begin{equation*}
\dim_{H}(\tilde{\Lambda}_m)\leq \frac{h(m, T)}{\lambda(m, T)}
\end{equation*}
of the first part of Theorem \ref{theorem-main}.  Here, we need some ideas of thermal-formalism method. We will make some delicate modification of the arguments in \cite{FLW2002}. Some key estimates there break down in our cases, and we need to be very careful  to deal with the hyperbolicity.

We define $$h_{k}(\alpha, \delta, s, \tau):=\sup\left\{h(\mu, T)-s\lambda(\mu, T): \mu\in
X(\alpha, \delta, k, \tau)
\right\}.$$ It is not difficult to prove that
$$s\mapsto h_{k}(\alpha, \delta, s, \tau)$$
is strictly decreasing in $[0, +\infty)$. Moreover,
\begin{equation}
h_{k}(\alpha, \delta, 0,\tau)\geq 0, \ \text{and} \ \lim_{s\rightarrow +\infty}h_{k}(\alpha, \delta, s,\tau)=-\infty.
\end{equation}
Thus, there exists unique $s_{*}=s_{*}(\alpha, \delta; k, \tau)\in [0, \infty)$ such that
\begin{equation}
h_{k}(\alpha, \delta, s_{*},\tau)=0.
\end{equation}
Indeed, we have
\begin{equation}
s_{*}=\sup\left\{\frac{h(\mu, T)}{\lambda(\mu, T)}: \mu\in X(\alpha, \delta, k, \tau)\right\}.
\end{equation}

Since the map
$s\mapsto h_{k}(\alpha, \frac{1}{j}, s, \tau)$ is decreasing and has a zero at $s=s_{*}$, by Proposition \ref{prop-key},  we have
$$t=g_{k}\Big(\alpha, \frac{1}{j}, s_{*}+\epsilon \Big)<0$$
for any $\epsilon>0$.
Thus, there exists $L_0=L_0(j)\geq j+1$,
for any $l\geq L_0$, we have
\begin{equation}
\limsup_{n\rightarrow\infty}\frac{1}{n}\log g_{k}\Big(\alpha, \frac{1}{j}; s_{*}+\epsilon, n, \frac{1}{l}\Big)<-\frac{t}{4}.
\end{equation}

It follows that there exists $U_0=U_0(l)\in \mathbb{N}$ such that for any $u\geq U_0$, we have
\begin{equation}
\sum_{[\omega]_n\in G_{k}(\alpha, \frac{1}{j}, u, \frac{1}{l})}(\diam I_u(\omega))^{s_{*}+\epsilon}\leq
\exp(-u\frac{t}{2}).
\end{equation}

Recalling that
$$
\tilde{\Lambda}_{m}=\left\{x\in  \Lambda: (A_{n_k})_{*}\delta_x
  \xrightarrow[n\rightarrow\infty]{*} m \ \text{for some} \  \{n_k\}_{k=1}^{\infty} \ \text{such that} \lim_{k\rightarrow\infty}n_k=\infty\right\},
$$
it is straightforward to see that
\begin{align*}
\tilde{\Lambda}_{m}\subset & \bigcup_{j=1}^{\infty}
\bigcup_{l=L_0(j)}^{\infty}\bigcap_{k=1}^{\infty}\bigcap_{n=1}^{\infty}
\bigcup_{u=n}^{\infty}G_{k}\left(\alpha, \frac{1}{j}, u, \frac{1}{l}\right)\\
\subset & \bigcup_{j=1}^{\infty}
\bigcup_{l=L_0(j)}^{\infty}\bigcap_{k=1}^{\infty}
\bigcup_{u\geq U_0}^{\infty}G_{k}\left(\alpha, \frac{1}{j}, u, \frac{1}{l}\right)
\end{align*}
We denote
\begin{equation*}
E_{j, k, l}=\bigcup_{u\geq U_0}^{\infty}G_{k}\left(\alpha, \frac{1}{j}, u, \frac{1}{l}\right).
\end{equation*}

Thus, we proved that
\begin{align*}
\mathcal{H}^{s_{*}+\epsilon}(E_{j, k, l})\leq & \sum_{u\geq U_0}\sum_{[\omega]_n\in G_{k}(\alpha, \frac{1}{j}, u, \frac{1}{l})}(\diam I_u(\omega))^{s_{*}+\epsilon}\\
\leq &
\sum_{u\geq U_0}\exp(-u\frac{t}{2})<+\infty.
\end{align*}
It follows that $\dim_{H} E_{j, k, l}\leq s_{*}$. And this implies that
$\dim_{H}\tilde{\Lambda}_m\leq s_{*}$, that is
\begin{align*}
\dim_{H}\tilde{\Lambda}_{m}\leq & \sup\left\{\frac{h(\mu, T)}{\lambda(\mu, T)}: \mu\in X(\alpha, \delta, k, \tau)\right\}\\
\leq & \sup \left\{\frac{h(\mu, T)}{\lambda(\mu, T)}: \mu\in X(\alpha,  k, \tau)\right\}
\end{align*}
where
\begin{equation*}
X(\alpha, k, \tau):=\left\{\mu\in \mathcal{M}(\Lambda, T): \left\|\int F_k d\mu-\alpha\right\|_{\infty}\leq \tau, \lambda(\mu, T)>0\right\}.
\end{equation*}

Noting that $\lambda(m, T)>0$, we have there exists $k_0\in \mathbb{N}$ such that
\begin{equation*}
\mathcal{M}_{m, k}\bigcap \mathcal{M}^{p}(\Lambda, T)=\emptyset
\end{equation*}
for any $k\geq k_0$ by Lemma \ref{lemma-parabolic}.

 Since both $\mathcal{M}_{m, k}$ and $\mathcal{M}^{p}(\Lambda, T)$ are compact, we know that there exists $\gamma>0$ such that
 \begin{equation}\label{eq-subtle}
 d(\nu_1, \nu_2)=\sum_{n=1}^{\infty}\frac{1}{2^{n}\|f_n\|}\Big|\int f_n d\nu_1-\int f_n d\nu_2\Big|\geq \gamma
 \end{equation}
 for any $\nu_1\in \mathcal{M}_{m, k}$, $\nu_2\in \mathcal{M}^{p}(\Lambda, T)$. Choose $k_*\in \mathbb{N}$ such that $k_{*}\geq k_0$ and $\frac{1}{2^{k_{*}-1}}< \gamma$.

We claim that for any $k\geq k_{*}$, there exists $\tau_0=\tau_0(k)$ such that for any $\tau\in (0, \tau_0)$
we have $\tilde{X}(\alpha, k, \tau)=X(\alpha, k, \tau)$, where
\begin{equation*}
\tilde{X}(\alpha, k, \tau):=\left\{\mu\in \mathcal{M}(\Lambda, T): \left\|\int F_k d\mu-\alpha\right\|_{\infty}\leq \tau\right\}.
\end{equation*}
Otherwise, we know that $X(\alpha, k, \tau)\subsetneq \tilde{X}(\alpha, k, \tau)$. Thus, there exists a sequence $\{\tau_n\}_{n=1}^{\infty}$ with $\lim\limits_{n\rightarrow\infty}\tau_n=0$ and invariant measures $\{\mu_n\}$ in $\mathcal{M}(\Lambda, T)$ with $\lambda(\mu_n, T)=0$ such that
$$\left\|\int F_k d\mu_n-\int F_k dm\right\|_{\infty}=\left\|\int F_k d\mu-\alpha\right\|_{\infty}\leq \tau_n.$$
We could assume that $\lim\limits_{n\rightarrow\infty}\mu_n=\nu$ for some $\nu \in \mathcal{M}(\Lambda, T)$ up to a sub-sequence.
It follows that
\begin{equation*}
\lambda(\nu, T)=\lim_{n\rightarrow\infty}\lambda(\mu_n, T)=0\
\text{and}\ \int F_k d\nu-\int F_k dm=0.
\end{equation*}
Hence, we get
\begin{align*}
d(\nu, m)=&\sum_{n=k+1}^{\infty}\frac{1}{2^{n}\|f_n\|}\Big|\int f_n d\nu-\int f_n d m\Big|\\
\leq & \sum_{n=k+1}^{\infty}\frac{1}{2^{n-1}}\leq \frac{1}{2^{k_{*}-1}}<\gamma.
\end{align*}

It is contradictory to the fact that $d(\nu, m)\geq \gamma$ by \eqref{eq-subtle} since
 $\nu\in \mathcal{M}^{p}(\Lambda, T)$ and $m\in \mathcal{M}_{m, k}$. Since $\tau$ is arbitrary, for any $k\ge k_*$, it is easy to prove that
\begin{align*}
\dim_{H}\tilde{\Lambda}_{m}\leq &\sup \left\{\frac{h(\mu, T)}{\lambda(\mu, T)}: \int F_k d\mu=\alpha\right\}\\
=&\sup \left\{\frac{h(\mu, T)}{\lambda(\mu, T)}: \int F_k d\mu=\alpha, \lambda(\mu, T)>0\right\}\\
=&\sup \left\{\frac{h(\mu, T)}{\lambda(\mu, T)}: \int F_k d\mu=\int F_k dm, \lambda(\mu, T)\geq \delta_{*}\right\}
\end{align*}
for some $\delta_{*}=\delta_{*}(m)$.
Noting that $$\bigcap_{k=k_*}^{\infty}\left\{\mu\in \mathcal{M}(\Lambda, T):\int F_k d\mu=\int F_k dm\right\}=\{m\},$$
it is easy to get
\begin{equation*}\label{eq-upperbound}
\dim_{H}(\tilde{\Lambda}_m)\leq \frac{h(m, T)}{\lambda(m, T)}.
\end{equation*}
\end{proof}

For any $m$  in $\mathcal{M}(\Lambda, T)$ and  $R>0$, we define
\[B_d(m,R):=\{\mu\in \mathcal{M}(\Lambda, T):
 d(\mu, m)\leq R\ \text{and}\ \lambda(\mu, T)>0\}.\]

We can prove  the following result by almost the same method presented here with minor changes. And we omit its proof here. We point out this result is essential to the proof of Theorem \ref{thm-uni-parabolic}.
\begin{prop}\label{prop-key-2}
Let $m$ be an invariant measure in $\mathcal{M}(\Lambda, T)$ such that $\lambda(m, T)>0$ and $B_d(m,R)$ as above. Furthermore let
\begin{equation*}
\mathcal{A}_{r}:=\left\{x\in \Lambda, Asym(\{(A_n)_{*}\delta_x\}_{n=1}^{\infty})\cap B_{d}(m, r)\neq \emptyset \right\}
\end{equation*}
for any $r\in (0, R)$. Then we have
\begin{equation*}
\dim_{H}\mathcal{A}_{r}\leq \sup\left\{\frac{h(\mu, T)}{\lambda(\mu, T)}: \mu\in B_{d}(m, r)\right\}.
\end{equation*}
\end{prop}

\section{Moran constructions}
Since the proof of the lower bound of  Theorem \ref{theorem-main} relies heavily on the constructions of Moran sets, which can be seen as the generalized cantor sets with  product structures.

We first establish a framework of abstract construction of the Moran set, which may be of independent interest. Indeed this kind of techniques have been greatly used in a certain amount of literature, for example see \cite{FLW2002, FLP, Chung2010}, however, it seems that there is no unified framework.
\subsection{Basic settings in Moran Construction}

Let $\mathcal{I}=\{I_{n, j}:  1\leq j\leq m_n, n\in \mathbb{N}\}$ be a sequence of the intervals
in $[0,1]$ such that
\begin{enumerate}
\item $int I_{n, i}\bigcap int I_{n, j}=\emptyset$ for any $n\in \mathbb{N}$ and $1\leq i\neq  j\leq m_n$;
\item  $I_{n, i}$ contains at least one element  $I_{n+1, j}$ for any $n\in \mathbb{N}$, and $1\leq i\leq m_n$;
\item for each $n\in \mathbb{N}$,  $I_{n+1, j}$ is contained in one of the elements in $I_{n, i}$ for $1\leq j\leq m_{n+1}$.
\end{enumerate}

We denote $r_{n}=\min\limits_{1\leq i\leq m_n} \diam (I_{n, i})$ and $R_n=\max\limits_{1\leq i\leq m_{n}} \diam(I_{n, i})$. We  assume  that  $\lim\limits_{n\rightarrow\infty} R_n=0$. Let $Y:=\bigcap\limits_{n=1}^{\infty}\bigcup\limits_{i=1}^{m_n} I_{n, i}$. Here, $I_{n, i}$ is said to be the {\bf $i$-th fundamental interval  in the $n$-th level} of $Y$. We will call $Y$ to be Moran set. For any $x\in Y$, there exists a sequence of nested intervals $\{I_{n, l_n(x)}\}_{n=1}^{\infty}$ such that  $x=\bigcap\limits_{n=1}^{\infty} I_{n, l_n(x)}$.

The following observation is crucial in the estimates of the lower bound of the Hausdroff dimension of  the Moran set.

\begin{lemma}\label{lemma-Moran}
Assume that the fundamental intervals satisfy the following conditions
\begin{equation}\label{regualr-growth-1}
\begin{split}
\lim_{n\rightarrow \infty} \frac{\log R_{n}}{\log r_n}=1  \quad \text{and} \quad
\lim_{n\rightarrow \infty} \frac{\log r_{n+1}}{\log r_n}=1.
\end{split}
\end{equation}
Let $\eta$ be a measure supported on $Y$ which satisfies the balanced property
\begin{equation}\label{regular-growth-2}
\lim\limits_{n\rightarrow\infty}\frac{\min\limits_{1\leq i\leq m_n} \log \eta(I_{n, i})}{\max\limits_{1\leq i\leq m_n}\log \eta(I_{n, i})}=1.
\end{equation}
 Then we have
\begin{equation*}\label{eq-subsequence}
\begin{split}
\limsup_{r\rightarrow 0}\frac{\log \eta(B(x,r))}{\log r}=
\limsup_{n\rightarrow\infty}\frac{\log \eta(I_{n, l_n(x)})}{\log r_n}\\
\liminf_{r\rightarrow 0}\frac{\log \eta(B(x,r))}{\log r}=
\liminf_{n\rightarrow\infty}\frac{\log \eta(I_{n, l_n(x)})}{\log r_n}
\end{split}
\end{equation*}
for any $x\in Y$.

\end{lemma}

\begin{proof}
For any $x\in Y$,  $r>0$, there exists a nested sequence of intervals $\{I_{k, l_k(x)}\}_{k=1}^{\infty}$, such that $x\in I_{k, l_k(x)}$ for any $k\in \mathbb{N}$. And  there always exits a unique integer  $n$ such that $R_{n+1}<r<R_{n}$.

We assume that there exists $N=N(x, n, r)$ intervals in $\{I_{n,  j}\}_{j=1}^{m_n}$ which intersect with $B(x, r)$. It follows that
\begin{equation*}
N r_{n}\leq \underset{B(x, r)\cap I_{n, j}\neq \emptyset}\sum |I_{n, j}|\leq 4R_{n}.
\end{equation*}
Thus, we get $N\leq  \frac{4R_n}{r_n}$. Combining the fact that $B(x, r)$ contains at least one element in $\{I_{n+1, j}: 1\leq j\leq m_{n+1}\}$, we have
\begin{equation*}
\frac{\log \eta(B(x,r))}{\log r}\leq \frac{\min_{1\leq j\leq m_{n+1}}\log \eta(I_{n+1, j})}{\log R_n}
\end{equation*}
and
\begin{equation*}
\frac{\log \eta(B(x,r))}{\log r}\geq  \frac{\log R_n-\log r_n}{\log R_{n+1}}+\frac{ \max_{1\leq j\leq m_n}\log \eta(I_{n, j})}{\log R_{n+1}}.
\end{equation*}
By  \eqref{regualr-growth-1} and \eqref{regular-growth-2}, we proved the desired conclusion.
\end{proof}
It seems that the conditions in Lemma \ref{lemma-Moran} are too strong and awkward at the first glance. Indeed they are very useful  especially in estimating the lower bound of the  Hausdroff dimension in the Moran constructions.
\subsection{Moran construction driven by dynamical system}
In many applications, the fundamental  intervals in the different levels in  Moran construction was obtained by dynamical system. In this section, we try to introduce a framework for Moran construction driven by dynamical system.

For each $n\in \mathbb{N}\cup\{0\}$, let $\mathcal{F}_n:=\{\hat{I}_{n,j}: 1\leq j\leq s_n\}$ be a family of $s_n$ disjoint closed intervals in $[0, 1]$ such that
 for each $n\in \mathbb{N}$ there exists $l_n\in \mathbb{N}$ such that $T^{l_n}(\hat{I}_{n, j})\supset \bigcup\limits_{j=1}^{s_{n+1}}\hat{I}_{n+1, j}$
    for any $j\leq s_{n}$.

Set  $n_i=\sum\limits_{j=1}^{i}l_j$ for $i>0$. We denote $\mathcal{I}_{k}$
be the family of sub-intervals such that for each $J\in \mathcal{I}_k$,  $T^{n_k} \  \text{is bijective from} \  J \ \text{to}  \ \hat{I}_{k, j} \ \text{for some} \ j\leq s_{k}$, and $T^{n_l}(J)\in \mathcal{F}_l $  for $0\leq l\leq k$. Indeed, $\mathcal{I}_{k}$ is exactly {\bf the set of all fundamental intervals of level $k$} and $\mathcal{I}=\bigcup_{k=1}^{\infty}\mathcal{I}_k$ is {\bf the set of all fundamental intervals of different levels. }

 It is easy to see that $\{\mathcal{I}_{l}\}_{l=0}^{\infty}$ is a nested sequence of compact set. Denote $\mathcal{Y}:=\bigcap_{n=1}^{\infty}\mathcal{I}_{n}$. Let $\mu_i$ be a sequence of measures on $\mathcal{F}_i$ and $J_n$ be a interval in $\mathcal{I}_n$.  For any $0\leq k\leq n$,  there  exits $t_k=t_k(J_n)\leq s_k$ such that
\begin{equation*}
T^{n_k}(J_n)\subset \hat{I}_{k, t_k}.
\end{equation*}
In other  words, we have
$$J_n=\bigcap_{ k=0}^{n}T^{-n_k}(\hat{I}_{k, t_k}).$$
We can introduce a family of measures $\eta_n$ on $\mathcal{I}_n$ such that
\begin{equation}
\eta_n(J_n)=\prod_{k=0}^{n}\mu_k(\hat{I}_{k, t_k}).
\end{equation}

It is easy to see the family of measures $\eta_n$ are compatible to each other in the following sense,  for any $J_n\in \mathcal{I}_n$,
\begin{equation}
\eta_m(J_n)=\eta_n(J_n)
\end{equation}
for any $m\geq n$. It is  standard to get a measure $\eta$ on $\mathcal{Y}$  by just taking the weak-$*$ limit of $\eta_n$. It follows that
 \begin{equation}
\eta(J_n)=\prod_{k=0}^{n}\mu_k(\hat{I}_{k, t_k}).
\end{equation}
for any $J_n\in \mathcal{I}_n$.

 For any $x\in \mathcal{Y}$, any $k\in \mathbb{N}\bigcup\{0\}$, there exists $J_k=J_k(x)$ in $\mathcal{I}_k$ and $\hat{I}_{k, t_k}=\hat{I}_{k, t_k}(x)$ in $\mathcal{F}_k$ such that
$$ x\in \bigcap\limits_{k=0}^{\infty} J_k$$ and
$T^{n_k}(x)\in \hat{I}_{k, t_k}$.

We can transfer Lemma \ref{lemma-Moran} to the following result, which will be very helpful in the construction of Moran set.
\begin{lemma}\label{lemma-Moran-Piece}
Let $\eta$ satisfies the balanced  property defined in Lemma \ref{lemma-Moran}. If the following asymptotic additive property
\begin{equation}
\lim_{n\rightarrow\infty}\frac{\log \diam J_n(x)}{\sum\limits_{i=0}^{n}\log \diam I_{i, t_i}(x) }=1
\end{equation}
and the tempered growth property
\begin{equation}
\lim_{n\rightarrow\infty}\frac{\log \diam I_{n+1, t_{n+1}(x)} }{\sum\limits_{i=0}^{n+1}\log \diam I_{i, t_i} (x)}=1
\end{equation}
hold,  we have
\begin{equation*}
\begin{split}
\limsup_{r\rightarrow 0}\frac{\log \eta(B(x,r))}{\log r}=
\limsup_{i\rightarrow\infty}\frac{\log \mu_i(I_{i, t_i}(x))}{\log \diam I_{i, t_i}(x)}\\
\liminf_{r\rightarrow 0}\frac{\log \eta(B(x,r))}{\log r}=
\liminf_{i\rightarrow\infty}\frac{\log \mu_i(I_{i, t_i}(x))}{\log \diam I_{i, t_i}(x)}
\end{split}
\end{equation*}
for any $x\in Y$.
\end{lemma}

The asymptotic additive property and tempered growth property  are crucial in the construction of Moran set and Moran measure which is a Bernoulli-like measure on the Moran set.  By Lemma \ref{lemma-Moran-Piece}, we only need to gluing a sequence of  subsystems carefully in the construction.

\subsection{Preparation of geometric  Moran construction}
In order to prove Theorem \ref{theorem-main}, we need to introduce two Moran sets, while the constructions in the both of the Moran sets share some similar steps in the initial constructions.

We first try to deal with the common steps in the two constructions a hyperbolic measure $\mu$. In the first case, we will set $\mu=m$ in Section \ref{sect-first-Moran-set}, where $m$ is the given hyperbolic measure in Theorem 1. In the second case, $\mu$ is an arbitrary given hyperbolic measure in Section \ref{sect-second-Moran-set}.
It is very convenient to transfer most of our discussion in the symbolic space.
Let $\{\epsilon_i\}_{i=1}^{\infty}$ be a decreasing sequence  such that $\lim\limits_{i\rightarrow\infty}\epsilon_i=0$.  Recall that $f$ and $g=\log |T'|$ are uniformly continuous on $\Lambda$,  there exists $k_i\in \mathbb{N}$   such that
\begin{equation}\label{var}
\begin{cases}
\underset{1\leq j\leq i}{\var_n} A_n f_j\circ \pi<\epsilon_i, \quad
\var_n A_n g\circ \pi<\epsilon_i,\\
 |-\frac{\log D_n(\omega)}{n}-A_n g(\pi(\omega))|<\epsilon_i \ (\text{for any} \ \omega\in\Sigma)
\end{cases}
\end{equation}
for any $n\geq k_i$.
Noting that $\mu$ is an ergodic measure, we have for $\mu$ a.e $\omega$,
\begin{equation}\label{block-1}
\begin{cases}
A_n f_j (\pi (\omega))\to \int f_j\circ \pi d \mu \ \text{for} \  1\leq j\leq  i,  \\
A_n g(\pi(\omega))\to \lambda(\mu,\sigma),  \\
-\frac{\log \mu[\omega|_n]}{n}\to h(\mu,\sigma).
\end{cases}
\end{equation}
 by Birkhoff's ergodic theorem and Shannon-Mcmillan-Breiman's theorem.
  For any $\delta>0$,  there exists a compact set $\Omega'(i)\subset \Sigma$ such that
  \begin{equation}
  \mu(\Omega'(i))>1-\delta
  \end{equation} and \eqref{block-1}
 holds uniformly on $\Omega'(i)$ by Egorov's theorem. Then there  exists   $m_i\geq k_i$ such that  for any  $n\geq m_i$  and any $\omega\in\Omega'(i)$, we
have
\begin{equation}\label{estimation}
\begin{cases}
\underset{1\leq j\leq i}\sup|A_n f_{j}(\pi(\omega))- \int f_{j} d \mu|<\epsilon_i, \\
|A_n g(\pi(\omega))- \lambda(\mu,\sigma)|< \epsilon_i, \\
|-\frac{\log\mu[\omega|_n]}{n} - h(\mu,\sigma)|< \epsilon_i,
\end{cases}
\end{equation}

In this way, we get a good block of  length at least $m_i$ with nice statistical behaviour.
A very naive idea is to product the blocks we have got in each step to construct the Moran set and introduce a product measure on itself. While, this construction is not really good since it is possible that
$$
\limsup\limits_{i\rightarrow\infty}
\frac{m_{i+1}}{m_{i}}=\infty,
$$
 and this makes it difficult to study the statistical analysis of the points in the Moran set.  To overcome this difficulty, we  introduce a new sequence with  very slow growth and  the corresponding sequence of blocks   still have well controlled statistical behaviors.

 We can assume that $m_{i+1}\geq 2 m_i$.
 The key point in the construction of Moran set is to construct the a suitable sequence $\{l_{i, j}:0\leq j\leq p_i, i\in \mathbb{N}\}$ such that
\begin{enumerate}
\item $p_{i}=m_{i+1}-m_{i}-1$;
\item $l_{i, j}=m_{i}+j$.
\end{enumerate}

\noindent It is a simple observation that $l_{i, p_i}=m_{i+1}-1$ and
\begin{equation}\label{slow-growth}
\lim_{i\rightarrow\infty}\sup\limits_{0\leq j< p_i}\frac{l_{i, j+1}}{l_{i, j}}=1.
\end{equation}
\section{Proof for the lower bound}
In this section, we will prove the lower bound for Theorem \ref{theorem-main}. And we will separate the proof into two cases.

\subsection{The construction of  the first Moran set}\label{sect-first-Moran-set}
Let $m$ be a hyperbolic measure.
Let $$
\Sigma(i,  j)=\{\omega_1\omega_2\dots\omega_{l_{i,j}}\ |\ \omega=(\omega_n)_{n=1}^{\infty}\in \Omega'(i)\}
$$
and  $${\Omega}(i, j)=\{\omega\in \Sigma: \omega_1\omega_2\dots\omega_{l_{i, j}}\in \Sigma(i, j)\}.$$

Now, we will use the framework discussed in {\bf Section} 4.3 and $\mu=m$. By the construction in {\bf Section 4.3}, we have
$$
\mu({\Omega}(i, j))\ge \mu(\Omega^\prime(i))\geq 1-\delta
$$
 for $1\leq j\leq p_{i}$.

We define the concatenation of $\Sigma(i, 0), \Sigma(i, 1) \dots$, and $\Sigma(i, p_i)$ as follows.
$$\overset{p_i}{\underset{j=0}{\prod}}\Sigma(i,  j):=\left\{\omega_1\omega_2\dots\omega_{l_{i, 0}}\omega_1\omega_2\dots\omega_{l_{i, 1}},\dots, \omega_1\omega_2\dots\omega_{l_{i, p_i}}\right\}.$$
Similarly, we define the  geometric Moran construction in the following,
$$
M:=\overset{\infty}{\underset{i=1}{\prod}}
\overset{p_i}{\underset{j=0}{\prod}}\Sigma(i,  j).
$$
(See figure 2 for the illustration of the construction of the Moran set.)
\begin{figure}
\begin{center}
\begin{tikzpicture}[scale=2.5]
    \draw[fill=blue!25] (0,0) rectangle (1,0.2);
    \draw[fill=blue!25] (1,0) rectangle (2,0.2);
    \draw[fill=blue!25] (3,0) rectangle (4,0.2);
    \node[above]at(2.5,0){$\dots$};
    \node at(0.5,0.1){$\Sigma(l,0)$};
    \node at(1.5,0.1){$\Sigma(1,1)$};
    \node at(3.5,0.1){$\Sigma(1,p_1)$};

    \draw[fill=blue!25] (0,-0.9) rectangle (1,-0.7);
    \draw[fill=blue!25] (1,-0.9) rectangle (2,-0.7);
    \draw[fill=blue!25] (3,-0.9) rectangle (4,-0.7);
    \node[above]at(2.5,-0.9){$\dots$};
    \node at(0.5,-0.8){$\Sigma(2,0)$};
    \node at(1.5,-0.8){$\Sigma(2,1)$};
    \node at(3.5,-0.8){$\Sigma(2,p_2)$};

    \draw[fill=blue!25] (0,-2.1) rectangle (1,-1.9);
    \draw[fill=blue!25] (1,-2.1) rectangle (2,-1.9);
    \draw[fill=blue!25] (3,-2.1) rectangle (4,-1.9);
    \node[above]at(2.5,-2.1){$\dots$};
    \node at(0.5,-2.0){$\Sigma(k,0)$};
    \node at(1.5,-2.0){$\Sigma(k,1)$};
    \node at(3.5,-2.0){$\Sigma(k,p_k)$};

    \draw[-stealth,rounded corners] (4,0.1) --(4.2,0.1)-- (4.2,-0.35)--(-0.2,-0.35)--(-0.2,-0.8)-- (0,-0.8);

    \draw[rounded corners] (4,-0.8)-- (4.2,-0.8)--(4.2,-1.05);
    \draw[-stealth,rounded corners] (4.2,-1.4)--(4.2,-1.55)-- (-0.2,-1.55)--(-0.2,-2.0)-- (0,-2.0);

    \node at(4.2,-1.2){$\vdots$};
    \draw[-stealth,rounded corners] (4,-2.0) --(4.2,-2.0)-- (4.2,-2.45)--(-0.2,-2.45)--(-0.2,-2.9)-- (0,-2.9);
    \node[above]at(0.2,-2.9){$\dots$};

      \draw[decorate,decoration={brace,mirror}] (0,0) -- (1,0) node[midway,below]{$l_{10}$};
      \draw[decorate,decoration={brace,mirror}] (1,0)-- (2,0) node[midway,below]{$l_{11}$};
      \draw[decorate,decoration={brace,mirror}] (3,0) -- (4,0) node[midway,below]{$l_{1p_1}$};

      \draw[decorate,decoration={brace,mirror}] (0,-0.9) -- (1,-0.9) node[midway,below]{$l_{20}$};
      \draw[decorate,decoration={brace,mirror}] (1,-0.9)-- (2,-0.9) node[midway,below]{$l_{21}$};
      \draw[decorate,decoration={brace,mirror}] (3,-0.9) -- (4,-0.9) node[midway,below]{$l_{2p_2}$};

    \draw[decorate,decoration={brace,mirror}] (0,-2.1) -- (1,-2.1) node[midway,below]{$l_{k0}$};
      \draw[decorate,decoration={brace,mirror}] (1,-2.1)-- (2,-2.1) node[midway,below]{$l_{k1}$};
      \draw[decorate,decoration={brace,mirror}] (3,-2.1) -- (4,-2.1) node[midway,below]{$l_{kp_k}$};
\end{tikzpicture}
\end{center}
\caption{Construction of the Moran set}
\end{figure}
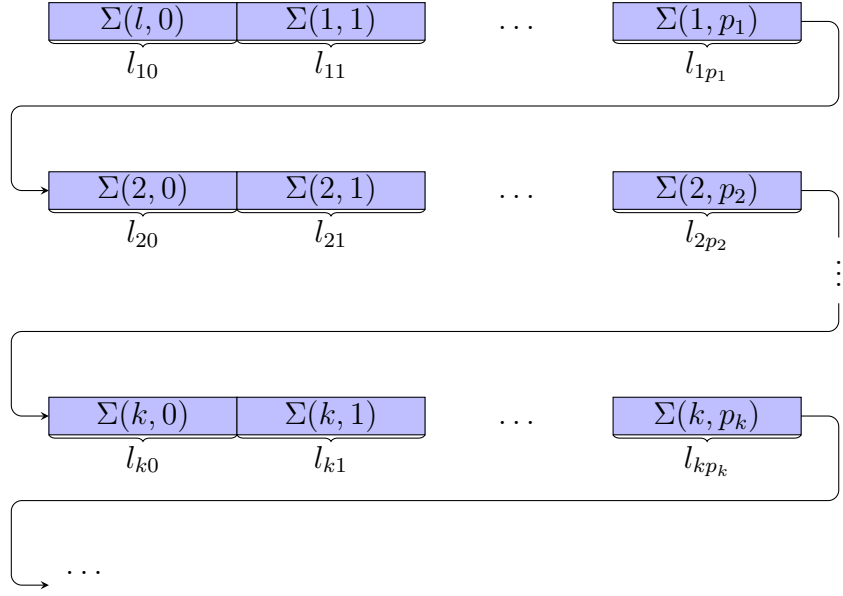
We  relabel the  sequences $\{l_{i, j}:0\leq j\leq  p_i, i\in \mathbb{N}\}$ and $\{\Sigma (i, j):0\leq j\leq  p_i, i\in \mathbb{N}\}$
 by $\{l_{i}^{*}\}_{i=1}^{\infty}$ and  $\{\Sigma^{*}(i)\}_{i=1}^{\infty}$ for  convenience. Correspondingly, we also use the notations $\{\Omega_{i}^{*}\}$ and $\{\epsilon_{i}^{*}\}_{i=1}^{\infty}$.
 It follows from \eqref{slow-growth}  that
\[
\lim_{i\rightarrow\infty}\frac{l_{i+1}^{*}}{l_{i}^{*}}=1.
\]
By the Stolz's theorem, we have
\begin{equation}\label{slow-growth-3}
\lim\limits_{n\rightarrow\infty}\frac{l_{1}^{*}+l_{2}^{*}
+\dots l_{n+1}^{*}}{l_{1}^{*}+l_{2}^{*}+\dots l_{n}^{*}}=1.
\end{equation}

\begin{lemma}\label{lemma-level}
For any $i\in \mathbb{N}$ and $\omega\in M$, we have $$\lim\limits_{n\rightarrow\infty} A_n f_i(\pi(\omega))=\int f_i d \mu.$$
\end{lemma}
\noindent {\bf Proof of Lemma \ref{lemma-level}}\
We  denote $n_{q}=\sum_{j=1}^{q}l_{j}^{*}$ for  any $q\in \mathbb{N}$. By \eqref{slow-growth-3}, we have
$$\lim\limits_{q\rightarrow\infty}\frac{n_{q+1}}{n_{q}}=1.$$
It is easy to see that
\[
\lim_{q\rightarrow\infty} A_{n_q}f_i(\pi(\omega))=\int f_i d \mu
\]
for each $i\in \mathbb{N}$.
For  $n_q\leq n<n_{q+1}$, we have
\begin{align*}
&\frac{1}{n}\bigg|S_n f_i(\pi(\omega))-n\int f_i d \mu\bigg|\\
\leq &
\frac{1}{n}\bigg|S_{n_q} f_i(\pi(\omega))-n_q\int f_i d \mu\bigg|+\frac{1}{n}\bigg|S_{n-n_q}f_i(\sigma^{n-n_q}\pi(\omega))
-(n-n_q)\int f_i d \mu\bigg|\\
\leq & \frac{1}{n}\bigg|S_{n_q} f_i(\pi(\omega))-n_q\int f_i d \mu\bigg|+\frac{2(n-n_q)\|f_i\|}{n}\\
\leq  & \frac{1}{n_q}\bigg|S_{n_q} f_i(\pi(\omega))-n_q\int f_i d \mu\bigg|+\frac{2(n_{q+1}-n_{q})\|f_i\|}{n_{q}}.
\end{align*}
Taking $n$ goes to infinity, we complete the proof of Lemma \ref{lemma-level}.
\hfill $\square$

 By Lemma \ref{lemma-level}, it is easy to check that $\pi (M)\subset \Lambda_{\mu}$.  Now, we will construct a probability measure $\eta$ on $M$, and we call it Moran measure. For each $w\in \Sigma^*(i)$, we define
\begin{equation}\label{construction-measure-1}
\rho^i_w:=\frac{\mu([w])}{\mu (\Omega^*(i))}.
\end{equation}
It is seen that $\sum_{w\in \Sigma^*(i)}\rho^i_w=1.$
Let
$
{\mathcal C}_n:=\{[w]:w\in \prod_{i=1}^n \Sigma^\ast(i)\}.
$ For each $w=w_1\cdots w_n\in {\mathcal C}_n$,  we define
\begin{equation}\label{construction-measure-2}
\nu([w]):=\prod_{i=1}^n \rho^i_{w_i}.
\end{equation}
 We still denote the extension to Borel $\sigma$-algebra $\sigma({\mathcal C}_n: n\ge 1)$ of $M$ by $\nu$. Let $\eta=\pi_{*}\nu$ and $\{I_{n, j}: 1\leq j\leq k_n\}=\pi (\prod\limits_{i=1}^{n}\Sigma^{*}(i))$.

We will split the remaining part of the proof into two steps. In this first step, we are going to prove the balanced property for $\eta$. For each $I_{n, j}$, there exists $\omega=\omega_1\omega_2\dots \omega_n\in \mathcal{C}_n$ such that $\pi(\omega)=I_{n, j}$. By the definition of $\eta$, we know that $\eta(I_{n, j})=\prod\limits_{i=1}^{n} \frac{\mu([w_i])}{\mu (\Omega^*(i))}$.

It follows from  \eqref{estimation}, we have
\begin{equation*}\label{entropy-moran-measure}
-\sum_{i=1}^{n}l_i(h(\mu, \sigma)+\epsilon_i)\leq \log \eta(I_{n, j})
\leq -n\log(1-\delta)-\sum_{i=1}^{n}l_i(h(\mu,\sigma)-\epsilon_i).
\end{equation*}
Noting that the both sides of the above inequality are independent of the  index $j$, we obtain
\begin{equation*}
\frac{\min\limits_{1\leq i\leq k_n} \log \eta(I_{n, i})}{\max\limits_{1\leq i\leq k_n}\log \eta(I_{n, i})}\geq \frac{ n\log(1-\delta)+\sum_{i=1}^{n}l_i(h(\mu, \sigma)-\epsilon^*_i)}
{\sum_{i=1}^{n}l_i(h(\mu, \sigma)+\epsilon^*_i)}
\end{equation*}
and
\begin{equation*}
\frac{\min\limits_{1\leq i\leq k_n} \log \eta(I_{n, i})}{\max\limits_{1\leq i\leq k_n}\log \eta(I_{n, i})}\leq \frac{\sum_{i=1}^{n} l_i(h(\mu, \sigma)+\epsilon^*_i)}
{n\log(1-\delta)-\sum_{i=1}^{n}l_i(h(\mu, \sigma)-\epsilon_{i}^{*})}.
\end{equation*}
This proves the balanced property for $\eta$ as $n$ goes to infinity.

In the second step,  we will  show that the fundamental intervals $\{I_{n, k}: 1\leq  k\leq k_n, n\in \mathbb{N}\}$ satisfy the tempered growth property. By Lemma \ref{lemma-growth-1} we can get
\begin{align*}
\log \diam I_{n, k}&=\log \diam D_{\sum_{j=1}^{n}l_{j}^{*}}(\omega)\\
&\leq -(S_{\sum_{j=1}^{n}l_{j}^{*}}g(\pi(\omega))-\sum_{j=1}^{n}
l^{*}_j\epsilon^*_n)\\
&= -\sum_{j=1}^{n}S_{l^*_j}g(\pi(\sigma^{l^*_1+\dots l^*_{j-1}}\omega))+\sum_{j=1}^{n}
l^{*}_j\epsilon^*_n\\
&\leq -\sum_{j=1}^{n}
l^{*}_j\lambda(\mu, \sigma)+\sum_{j=1}^{n}
l^{*}_j(\epsilon^*_n+\epsilon^*_j).
\end{align*}
Thus, we have
\begin{equation}\label{diameter-upper-bound}
\log R_n\leq -\sum_{j=1}^{n}
l^{*}_j\lambda(\mu, \sigma)+\sum_{j=1}^{n}
l^{*}_j(\epsilon^*_n+\epsilon^*_j)
\end{equation}
and
\begin{equation}\label{diameter-lower-bound}
\log r_n\geq -\sum_{j=1}^{n}
l^{*}_j\lambda(\mu, \sigma)-\sum_{j=1}^{n}
l^{*}_j(\epsilon^*_n+\epsilon^*_j).
\end{equation}
By the Stolz's theorem, it is not difficult to see that
\begin{equation}
\begin{split}
\lim_{n\rightarrow \infty} \frac{\log R_{n}}{\log r_n}=1 \quad \text{and} \
\lim_{n\rightarrow \infty} \frac{\log r_{n+1}}{\log r_n}=1.
\end{split}
\end{equation}
It follows from  \eqref{entropy-moran-measure}, \eqref{diameter-upper-bound} and \eqref{diameter-lower-bound} that
\[
\lim_{n\rightarrow\infty}\max_{j}\Big|\frac{\log \eta(I_{n, j})}{\log \diam (I_{n, j})}-\frac{h(\mu, \sigma)}{\lambda(\mu, \sigma)}\Big|=0.
\]
By Lemma \ref{lemma-Moran}, we have
\[
\lim_{r\rightarrow 0}\frac{\log \eta(B(x,r))}{\log r}=\frac{h(\mu, \sigma)}{\lambda(\mu, \sigma)}
\]
for any $x\in \pi(M)$.
Recall that $\mu=m$, we obtain
\begin{equation}\label{eq-lowerbound-1}
\dim_{H}\Lambda_{m}\geq \dim_{H}\pi(M)\geq \frac{h(m, \sigma)}{\lambda(m, \sigma)}.
\end{equation}
Noting that $\dim_{H}\Lambda_{m}\leq \dim_{H}\tilde{\Lambda}_m$, we have
$$\dim_{H}\Lambda_{m}=\dim_{H}\tilde{\Lambda}_m=\frac{h(m, \sigma)}{\lambda(m, \sigma)}.$$
This also recover first part of Theorem B in our setting.

\subsection{The construction of second Moran set}\label{sect-second-Moran-set}
In order to prove the lower bound of the second part of Theorem \ref{theorem-main}, we need to make some modifications  of the Moran set $M$ constructed above.  Let $\mu$ be arbitrary hyperbolic measure.
Since $m$ is a parabolic ergodic measure, we assume that $m$
supports  on $\pi(1^{\infty})$ without loss of generality. There exists a sequence $\{k_i\}_{i=1}^{\infty}$
such that
\[
\lim_{i\rightarrow\infty}k_i=\infty, \quad \lim_{i\rightarrow\infty}\frac{k_{i+1}}{k_i}=1, \quad \lim_{i\rightarrow\infty}k_i\epsilon_i=0.
\]
Let $k_{i, j}=k_i$ for $1\leq j\leq p_i$. Denote
$$
\Sigma(i, j)=\{\omega_1\omega_2\dots\omega_{l_{i,j}}1^{k_{i,j}l_{i,j}}\ |\ \omega\in \Omega'(i)\}
$$
and  $${\Omega}(i, j)=\{w:\omega_1\omega_2\dots\omega_{l_{i,j}}1^{k_{i,j}l_{i,j}}\in \Sigma(i, j)\}.$$
Next, we still use the framework discussed in {\bf Section} 4.3 and by the construction in {\bf Section 4.3}, we have
$$
\mu({\Omega}(i, j))\ge \mu(\Omega^\prime(i)\geq 1-\delta.
$$   We define a  new Moran set in the following,
$$
\hat{M}:=\overset{\infty}{\underset{i=1}{\prod}}
\overset{p_i}{\underset{j=0}{\prod}}\Sigma(i, j)
$$
and relabel the  sequences $\{l_{i, j}: 0\leq j\leq p_i, i\in \mathbb{N}\}$ and $\{\Sigma(i, j): 0\leq j\leq p_i, i\in \mathbb{N}\}$
 by $\{l_{i}^{**}\}_{i=1}^{\infty}$ and  $\{\Sigma^{**}(i)\}_{i=1}^{\infty}$. Analogously, we also take the relabelling sequences $\{\Omega^{**}(i)\}_{i=1}^{\infty}$, $\{k_{i}^{**}\}_{i=1}^{\infty}$ and $\{\epsilon_{i}^{**}\}_{i=1}^{\infty}$. For each  $\hat{w}=w 1^{k_{i}^{**}l_{i}^{**}}\in \Sigma^{**}(i)$, we define a probability measure on $\Sigma^{**}(i)$ by
\begin{equation}\label{construction-measure-1}
\rho^i_{\hat{w}}:=\frac{\mu([w 1^{k_{i}^{**}l_{i}^{**}}])}{\mu (\Omega^{**}(i))}.
\end{equation}
Denote
$
{\mathcal C}_n:=\{[w]:w\in \prod_{i=1}^n \Sigma^{**}(i)\}.
$ For each $w=w_1\cdots w_n\in {\mathcal C}_n$, we define
\begin{equation}\label{construction-measure-2}
\hat{\nu}([w]):=\prod_{i=1}^n \rho^i_{w_i}.
\end{equation}
 This measure can be uniquely extended to  $\hat{M}$ and we still denote it by $\hat{\nu}$.  Let  $\{\hat{I}_{n, j}: 1\leq j\leq k_n\}=\pi (\prod\limits_{i=1}^{n}\Sigma^{**}(i))$ and denote
 $\hat{\eta}=\pi_{*}\hat{\nu}$. We can
 get the following estimates,
\begin{equation*}
\begin{split}
-\sum_{i=1}^{n}l_{i}^{**}(h(\mu, \sigma)
+\epsilon^{**}_{i})
\leq \log \hat{\eta}(\hat{I}_{n, k})
\leq  -\sum_{i=1}^{n}l_{i}^{**}(h(\mu, \sigma)
-\epsilon^{**}_{i}),\\
\log \diam(\hat{I}_{n, k})\leq -
\sum_{j=1}^{n}l_{j}^{**}\lambda(\mu,
\sigma)+\sum_{j=1}^{n}l_{j}^{**}(1+k_{j}^{**})
\epsilon^{**}_{j},\\
\log \diam(\hat{I}_{n, k})\geq
-\sum_{j=1}^{n}l_{j}^{**}\lambda(\mu,
\sigma)-\sum_{j=1}^{n}l_{j}^{**}
(1+k_{j}^{**})\epsilon^{**}_{j}
\end{split}
\end{equation*}
for any $1\leq k\leq k_n$. It is almost the same to prove that $\pi(\hat{M})\subset \Lambda_{m}$ and
$\dim_{H}\pi(\hat{M})\geq \frac{h(\mu, T)}{\lambda(\mu, T)}$. Noting that $\mu$ is arbitrary, we have
\begin{equation}
\dim_{H}\Lambda_m\geq \sup_{\mu}\left\{\frac{h(\mu, T)}{\lambda(\mu, T)}: \lambda(\mu, T)>0, \mu \ \text{is ergodic}\right\}.
\end{equation}
We need the following lemma about the hyperbolic dimension of $\Lambda$.
\begin{lemma}[\cite{Urbanski1996}]\label{lemma-hyperbolic-dimension}
Under the setting above,  we have
\begin{equation}
\dim_{H}^{hyp}\Lambda=\sup_{\mu}\left\{\frac{h(\mu, T)}{\lambda(\mu, T)}: \lambda(\mu, T)>0, \mu \ \text{is ergodic}\right\}.
\end{equation}
\end{lemma}
By Lemma \ref{lemma-hyperbolic-dimension}, we obtain $\dim_{H}\Lambda_m\geq \dim_{H}^{hyp}\Lambda$. It is proved in \cite{GR20092} if $T$ is a $C^2$ map, we have
\[
\dim_{H}\Lambda=\sup_{\mu}\left\{\frac{h(\mu, T)}{\lambda(\mu, T)}: \lambda(\mu, T)>0, \mu \ \text{is ergodic}\right\}.
\]
  Hence, we get
$$\dim_{H}\Lambda_m=\dim_{H}\Lambda$$
if $T$ is a $C^2$ map. This recovers the second part of Theorem B.
Noting that $\dim_{H}\tilde{\Lambda}_m\geq \dim_{H}\Lambda$, we complete the second part of Theorem \ref{theorem-main}.

\section{Proof of Theorem \ref{thm-uni-parabolic}}}

In this section, we will prove Theorem \ref{thm-uni-parabolic}.
Recall that $g(x)=\log |T'(x)|$ for any $x\in \Lambda$.
By Theorem \ref{theorem-main}, we only need to prove that $\dim_{H}\Lambda_{\delta_p}=\dim_{H}\tilde{\Lambda}_{\delta_p}$ with the assumption
\begin{equation}\label{eq-assumption}
\dim^{hyp}_{H}\Lambda<\dim_{H}\Lambda.
\end{equation}


Let  $$\Lambda_{*}:=\{x\in \Lambda: \liminf\limits_{n\rightarrow\infty} A_n g(x)>0\}$$ and $$\Lambda^{*}:=\{x\in \Lambda: \limsup A_ng(x)>0\}.$$
It is proved in \cite{JJOP2010} that
\begin{equation}\label{eq-JJOP}
\dim^{hyp}_{H}\Lambda=\dim_{H}\Lambda_{*}.
\end{equation}
By \eqref{eq-JJOP} and \eqref{eq-assumption}, we know that
\begin{equation*}
\dim_{H}\Lambda=\dim_{H}\Lambda\backslash\Lambda_{*}.
\end{equation*}

 We write $\Lambda^*=\bigcup_{k=1}^{\infty}\Lambda_{k}^{*}$, where $\Lambda_{k}^{*}=\{x\in \Lambda^{*}: \limsup\limits_{n\rightarrow\infty} A_ng(x)>\frac{1}{k}\}$ for any $k\in \mathbb{N}$. For any $n\in \mathbb{N}$, we denote
$$C_{n}:=\left\{\mu\in \mathcal{M}(\Lambda, T): d(\mu, \delta_{p})\geq \frac{1}{n}\right\}.$$

Since $T$ has only one fixed point, we know that
$\lambda(\mu, T)>0$ for any $\mu\in C_n$. Noting that $C_n$ is compact, then for any positive real number $\rho_n<\frac{1}{n}$, there exists finitely many invariant measures $\{\mu_i\}_{i=1}^{l_n}
\subset C_n$ for some $l_n\in \mathbb{N}$ such that $C_n\subset \bigcup_{i=1}^{l_n} int B_{d}(\mu_i, \rho_n)$ and $\delta_{p}\notin \bigcup_{i=1}^{l_n} int B_{d}(\mu_i, \rho_n)$,
where $int B_{d}(\mu_i, \rho_n)$ is the interior of the closed ball $B_{d}(\mu_i, \rho_n)=\{\mu: d(\mu, \mu_i)\leq \rho_n\}$.

Let $$B_n:=\{x\in\Lambda: Asym(\{(A_m)_{*}\delta_x\}_{m=1}^{\infty})\cap C_n\neq \emptyset\}$$
and
$$
B_{n}^{i}:=\{x\in\Lambda: Asym(\{(A_m)_{*}\delta_x\}_{m=1}^{\infty})\cap B_{d}(\mu_i, \rho_n)\neq \emptyset\}
$$
for $1\leq i\leq l_n$.
It is easy to see that
\begin{equation}\label{eq-union}
\Lambda^{*}=\bigcup_{n=1}^{\infty}B_n=\Lambda\backslash\Lambda_{\delta_p}.
\end{equation}
By Proposition \ref{prop-key-2}, we have
$$
\dim_{H}B_{n}^{i}\leq \sup\left\{\frac{h(\mu, T)}{\lambda(\mu, T)}: \mu\in B_d(\mu_i, \rho_n)\right\}\leq \dim_{H}^{hyp}\Lambda.
$$
It follows that
$$\dim_{H}B_n\leq \max_{1\leq i\leq l_n}\dim_{H}B_{n}^{i}\leq \dim_{H}^{hyp}\Lambda.$$
Now, we get $$\dim_{H}\Lambda^{*}\leq \dim_{H}^{hyp}\Lambda.$$
Together with \eqref{eq-JJOP}, we have
$$\dim_{H}\Lambda^{*}=\dim_{H}^{hyp}\Lambda.$$
This implies that $\dim_{H}\Lambda_{\delta_p}=\dim_{H}\Lambda$. Since $$\dim_{H}\tilde{\Lambda}_{\delta_p}\geq \dim_{H}\Lambda_{\delta_p},$$
we get
\begin{equation}
\dim_{H}\tilde{\Lambda}_{\delta_p}=\dim_{H}\Lambda_{\delta_p}.
\end{equation}

\section{Acknowledgement}
We would like to thank Professor Shen Weixiao for drawing our attention to recovering the hyperbolicity in the non-uniformly hyperbolic interval maps along the subsequence. The research was partially supported  by National Key R\&D Program of China (2020YFA0713300) and NSFC of China (No.11771233, No.11901311).


\end{document}